\documentclass{amsart}
\usepackage{amsmath,amssymb,amsthm}

\newtheorem{theo}{Theorem}[section]
\newtheorem{defi}[theo]{Definition}
\newtheorem{lemm}[theo]{Lemma}

\theoremstyle{definition}
\newtheorem{rema}[theo]{Remark}
\newtheorem{exam}[theo]{Example}
\begin{document}

\title{Rohlin Flows on Amalgamated Free Product Factors}
\author{Koichi Shimada}
\email{shimada@ms.u-tokyo.ac.jp}
\address{Department of Mathematical Sciences
University of Tokyo, Komaba, Tokyo, 153-8914, Japan}
\date{}
\begin{abstract}
We construct flows on an amalgamated free product factor and characterize the Rohlin property for them by faithfulness of $\mathbf{R}\times \mathbf{Z}$ actions on a measured space.  As a corollary, we find Rohlin flows on this factor. This is the first example of Rohlin flows on non-McDuff factors. We also also apply Masuda--Tomatsu's theorem  to our Rohlin flows, and obtain a classification up to strong cocycle conjugacy. However, it also turns out that usual  cocycle conjugacy is different from strong cocycle conjugacy for our flows. Consequently, even though Masuda--Tomatsu's analysis is applicable for flows on non-McDuff factors, its effect for classification up to cocycle conjugacy is limited.
\end{abstract}

\maketitle
\section{Introduction}

In this paper, we construct Rohlin flows on amalgamated free product factors and consider a classification of them.

Studying group actions is one of the most interesting topics in operator algebra theory.  After Connes \cite{C}, \cite{C2} has completely classified single automorphisms of the approximately finite dimensional (hereafter abbreviated AFD) $\mathrm{II}_1$ factor up to cocycle conjugacy, the classification of group actions has been remarkably developed; discrete amenable group actions on AFD factors are completely classified by many hands \cite{KtST}, \cite{KwhST}, \cite{M}, \cite{O}, \cite{ST}, \cite{ST2}, and there has been great progress in the classification of compact group actions by Masuda--Tomatsu \cite{MT1}, \cite{MT2}, \cite{MT3}. Then our next interest is the classification of non-compact continuous groups, especially, $\mathbf{R}$-actions (flows). In the study of flows, a great deal of things remain to be done. One of the main topics related to flows is about ``outerness'' for flows. Defining the appropriate ``outerness'' for flows, constructing and classifying ``outer'' flows have been  studied by many people. The first studies of this type have been done by Kawahigashi \cite{Kwh1}, \cite{Kwh2}, \cite{Kwh3}. As in Theorem 1.6 of \cite{Kwh2}, Theorem 16 of \cite{Kwh3}, he constructed flows on the AFD $\mathrm{II}_1$ factor by various ways, and investigated when they are ``very outer'' (which is equivalent to being cocycle conjugate to an infinite tensor product type action with full Connes spectrum). Later, Kishimoto \cite{Ksm} defined the Rohlin property for flows on $C^*$-algebras as in the Introduction of \cite{Ksm}, which is a candidate for ``being very outer'' for flows. Kawamuro \cite{Kwm} defined the Rohlin property for flows on the AFD $\mathrm{II}_1$ factor (Definition 2.2 of \cite{Kwm}). Considering the Rohlin  property plays a very important role in classifications. Recently, Masuda--Tomatsu \cite{MT} established the classification theorem for the Rohlin flows (Theorem 5.14 of \cite{MT}). Remarkably, their theorem is applicable not only for flows on McDuff factors, but also for flows on arbitrary separable von Neumann algebras. However, there are no known examples of Rohlin flows on non-McDuff factors. So, constructing Rohlin flows on non-McDuff factors should lead to interesting observations (See also Problem 8.2. of \cite{MT}).  

\bigskip

The main result of this paper is constructing flows on an amalgamated free product factor (Hereafter abbreviated AFP) and characterize the Rohlin property for them (Theorem \ref{3.1.3}). This yields the first examples of Rohlin flows on non-McDuff factors. Besides constructions, classification of flows is very important.  We also classify Rohlin flows constructed in Theorem \ref{3.1.3}, up to strong cocycle conjugacy (Lemma \ref{4.1.2}), by using Masuda--Tomatsu's theorem. However, a main interest is a classification up to usual cocycle conjugacy. So, it is important to study the difference between cocycle conjugacy and strong cocycle conjugacy, and we show that they are strictly different for our Rohlin flows (Theorem \ref{4.2.2}). Consequently, even though Masuda--Tomatsu's analysis for flows based on the Rohlin property is applicable for flows on non-McDuff factors, its effect is limited and another idea seems to be needed to classify flows on non-McDuff factors.

\section{Preliminaries}
\label{sec2}
In this section, we collect basic facts about Rohlin flows and amalgamated free product factors. Here, $M$ always denotes a separable von Neumann algebra.

\subsection{The Rohlin property}

First, we recall the Rohlin property. Basic references are Chapter 5 of Ocneanu \cite{O}, Kishimoto \cite{Ksm}, Kawamuro \cite{Kwm} and Masuda--Tomatsu \cite{MT}. Let $\omega $ be a free ultrafilter on $\textbf{N} $. We denote by $l^{\infty }(M)$ the $\mathrm{C}^*$-algebra consists of all bouded sequences in $M$. Set
\[ I_{\omega }:=\{ (x_n) \in l^{\infty }(M) | \text{*strong-lim}_{n\to \omega }(x_n)=0 \}, \]
\begin{align*}
 N_{\omega }:=\{ (x_n) \in l^{\infty }(M) | &\text{for all } (y_n)\in I_{\omega }, \\
                                            &\text{ we have } (x_ny_n)\in I_\omega \text{ and }(y_nx_n)\in I_\omega \} ,
 \end{align*}
\[ C_{\omega }:=\{ (x_n) \in l^{\infty }(M) |\text{for all } \phi \in M_{*}, \lim _{n\to \omega }\parallel [\phi ,x_n] \parallel =0 \}. \]
Then we have $I_{\omega }\subset C_{\omega }\subset N_{\omega }$ and $I_{\omega }$ is a closed ideal of $N_{\omega}$. Hence we set $M^{\omega }:=N_{\omega }/I_{\omega }$. Denote the canonical quotient map $N_{\omega }\to M^{\omega }$ by $\pi $. Set $M_{\omega }:=\pi (C_{\omega })$. Then $M_{\omega }$ and $M^{\omega }$ are von Neumann algebras as in Proposition 5.1 of Ocneanu \cite{O}. Let $\alpha $ be an automorphism of $M$. We define an automorphism $\alpha ^{\omega }$ of $M^{\omega }$ by $\alpha^{\omega }((x_n))=(\alpha (x_n))$ for $(x_n)\in M^{\omega}$. Then we have $\alpha^{\omega }(M_{\omega })=M_{\omega }$. By restricting $\alpha ^{\omega }$ to $M_{\omega }$, we obtain an automorphism $\alpha _{\omega }$ of $M_{\omega }$. Hereafter we denote $\alpha ^{\omega }$ and $\alpha _{\omega }$ by $\alpha $ if there is no chance of confusion.

\bigskip 

 Choose a normal faithful state $\varphi $ on $M$. For a flow $\alpha $ on a von Neumann algebra $M$, set 
\begin{multline*}
 M_{\omega ,\alpha }:=\{ (x_n)\in M_\omega |\text{for all } \epsilon >0,\text{ there exists } \delta >0 \text{ such that } \\
      \{n\in \textbf{N}|\ \| \alpha _t(x_n)-x_n\|_{\varphi }^{\sharp } <\epsilon   \ \mathrm{ for }\ |t|<\delta \}\in \omega \}. 
\end{multline*}
Where, $\| x \| _{\varphi }^{\sharp }:=\sqrt{ \varphi (x^*x)+\varphi (xx^*) }$, which metrizes the *strong topology of the unit ball of $M$. This is a von Neumann subalgebra of $M_\omega $ and satisfies $\alpha _t(M_{\omega , \alpha})=M_{\omega , \alpha}$ as in Lemma 3.10 of Masuda--Tomatsu \cite{MT}.

\begin{defi}\textup{(Definition 2.2 of Kawamuro \cite{Kwm}, Definition 4.1 of Masuda--Tomatsu \cite{MT})}

We say that a flow $\alpha$ has the Rohlin property if for any $p\in \textbf{R}$, there exists a unitary $u\in U(M_{\omega,\alpha})$ with $\alpha_t(u)=e^{-ipt}u$ for $t\in\textbf{R}$.
\label{2.3.1}
\end{defi}

As shown in Masuda--Tomatsu \cite{MT}, a Rohlin flow is always centrally free and has full Connes spectrum. So it is reasonable to regard a Rohlin flow as a sufficiently free action.

For Rohlin flows, there is a classification theorem.

\begin{theo}\textup{(Theorem5.14 of Masuda--Tomatsu \cite{MT})}
For two Rohlin flows $\theta ^1$, $\theta ^2$, they are strongly cocycle conjugate if and only if $\theta^1_t \circ \theta^2_{-t}\in \overline{\mathrm{Int}}(M)$ for all $t$.
\label{2.3.4}
\end{theo}

Here, two flows $\theta ^1$ and $\theta ^2$ are said to be strongly cocycle conjugate if there exists a $\theta ^2$-cocycle $u$ and $\sigma \in \overline{\mathrm{Int}}(M)$ with $\mathrm{Ad} (u_t) \circ \theta ^2_t =\sigma \circ \theta ^1 _t \circ \sigma ^{-1}$. The condition that the automorphism $\sigma $ is approximately inner is just a technical condition. And our main interest is classification up to usual cocycle conjugacy. However, for automorphisms on AFD factors, approximate innerness is characterized by Connes--Takesaki module as in Theorem 1 of Kawahigashi--Sutherland--Takesaki \cite{KwhST}. So this theorem is very powerful for flows on AFD factors.  

\subsection{Amalgamated free products}
\label{2.4}

In order to construct Rohlin flows on non-McDuff factors,  examples of non-McDuff factors are important. One of the important classes of non-McDuff factors consists of amalgamated free product factors. Here, we list up facts about amalgamated free product factors needed for our construction. Basic references are Popa \cite{P0}, Popa \cite{P}, Ueda \cite{U1} and Ueda \cite{U2}.

Let $A$, $B$ be two von Neumann algebras with a common von Neumann subalgebra $D$. Suppose that there exist normal faithful conditional expectations $E_A:A \to D $, $E_B:B\to D$. Then as in 3 of Popa \cite{P0}, Definition 2.4, Theorem 2.5 of Ueda \cite{U1}, there exists a von Neumann algebra $M$ with the following properties.

\textup{(1)} There are normal faithful representations $\pi _A:A\to B(H)$, $\pi _B:B\to B(H)$ such that $\pi _A|_D=\pi _B|_D$.

\textup{(2)} $M=(\pi _A(A)\bigcup \pi _B(B)) ''$.

\textup{(3)} There exists a normal faithful conditional expectation $E:M\rightarrow \pi _A(D)=\pi _B(D)$ such that $E(\pi_A(x))=\pi_A(E_A(x))$, $E(\pi_B(x))=\pi_B(E_B(x))$.

\textup{(4)} The von Neumann algebras $A$ and $B$ are free with amalgamation over $D$, i.e., 
  \[ E(\pi _{C_1}(x_1)\cdots \pi _{C_n}(x_n))=0. \]
  for $n\in \mathbf{N}$, $x_i \in \mathrm{ker} E \circ \pi_{C_i}$, $C_i=A,B$ for $i=1, \cdots n$, $C_i\not =C_{i+1}$ for $i=1,\cdots n-1$.

Hereafter, we identify $\pi _A(A)$ with $A$, $\pi _B(B)$ with $B$, $\pi _A(D)=\pi _B(D) $ with $D$, respectively. Moreover, as in Theorem 2.5 of Ueda \cite{U1}, this $M$ is unique in the following sense: if $(\pi _A^1$, $\pi _B^1$, $E_1$, $M_1)$ and $(\pi _A^2$, $\pi _B^2$, $E_2$, $M_2)$ satisfy the above conditions \textup{(1)--(4)}, then there exists a unique *-homomorphism $\phi :M_1\to M_2$ with $\phi \circ \pi_A^1=\pi _A^2$, $\phi \circ \pi _B^1=\pi _B^2$, $\phi \circ E_1=E_2 \circ \phi$.

This $M$ is said to be the amalgamated free product of $A$ and $B$ relative to $D$ and  denoted by $A*_DB$ (Definition 3.3 of Popa \cite{P0}, Definition 2.4 of Ueda \cite{U1}).

This definition depends on the choice of $E_A$ and $E_B$. However, if $D$ is a common Cartan subalgebra of $A$ and $B$, $E_A$ and $E_B$ are unique. Hence, in this case, $M$ depends only on the choice of $A$, $B$ and $D$.  For our purpose, it is also important to note that if both $A$ and $B$ are of non-type I and if $D$ is a common Cartan subalgebra of $A$ and $B$, then $M$ is a non-McDuff factor. Moreover, we have $M_\omega =M'\cap M^\omega \subset D^\omega$, as shown in Theorem 8 of Ueda \cite{U2}.

\section{Construction of Rohlin flows}
\label{sec3}
Here we construct Rohlin flows on a non-McDuff factor. Let $D=L^{\infty}(X,\mu )$ be a diffuse separable abelian von Neumann algebra, where $\mu$ is a probability measure. Choose a free ergodic $\mu$-preserving action $\alpha:\textbf{Z} \curvearrowright D$. Then $A:= D\rtimes _\alpha \textbf{Z}\supset D$ is a pair of the AFD type $\text{II}_1$ factor and its Cartan subalgebra. There is a unique action $\alpha *\alpha:\textbf{F}_2 \curvearrowright D$ which satisfies $\alpha *\alpha(a)=\alpha$, $\alpha *\alpha(b)=\alpha$, where $a$, $b$ are two generator of $\textbf{F}_2$. Set $M:=A*_DA$. Then $M$ is isomorphic to $D\rtimes_{\alpha *\alpha}\textbf{F}_2$ because the latter satisfies conditions (1)--(4) of subsection \ref{2.4}, and this $M$ is a non-McDuff factor.

\begin{lemm}\textup{(See also Theorem 2.6 of Ueda \cite{U1})}

Let $\theta :\textbf{R}\curvearrowright D$ be a $\mu $-preserving flow commuting with $\alpha$. Let $\{u^{i}_t\}$ be $\theta $-cocycles \textup{(}$i=1,2$\textup{)}. Then the action $\theta$ extends to $M$ by $\theta_t(\lambda _a)=u^1_t \lambda _a$, $\theta_t(\lambda _b)=u^2_t\lambda _b$ for $t\in \textbf{R}$.\label{3.1.1}
\end{lemm}

\begin{proof}
Fix $t\in \textbf{R}$. Since $\theta _t$ commutes with $\alpha $, the injective homomorphisms $\pi_A:A\cong \{ \{ \lambda_a \} \cup D \} '' \hookrightarrow A*_DA$, $\pi_B:A\cong \{ \{ \lambda_b \} \cup D \} '' \hookrightarrow A*_DA$ satisfying the following are well-defined.
\[ \pi _A(\lambda _a)=u^1_t\lambda _a,\ \pi _B(\lambda _b)=u^2_t\lambda _b,\  \pi _A(x)=\pi _B(x)=\theta _t(x)\ \mathrm{for}\ x\in D. \]
Since this $(\pi _A, \pi _B)$ satisfies conditions (1)--(4) of subsection \ref{2.4}, there exists an automorphism $\theta_t$ of $A*_DA$ such that $\theta _t|_{\{ \{ \lambda _a \}  \cup D\} '' }=\pi _A$, $\theta _t|_{\{ \{ \lambda _a \}  \cup D\} '' }=\pi _B$. 
 It is not difficult to see that the map $t\mapsto \theta _t(y)$ is strongly continuous for $y\in M$.

\end{proof}

For the above flows, we give a characterization of the Rohlin property. 

In order to do this, we make use of the following Rohlin type theorem for $\mathbf{R}\times \mathbf{Z}$ actions on the standard probability space, which is a part of a theorem of Lind \cite{L} or Ornstein--Weiss \cite{OW}.

\begin{lemm}\textup{(Theorem 1 of Lind \cite{L})}
Let $R$ be a $\mu $ -preserving faithful ergodic action of $\textbf{R}\times \textbf{Z}$ on the standard probability space $(X,\mu )$. Then for any $\epsilon >0$, for any $N\in \textbf{N}$ and for any $T>0$, there exists a Borel subset $Y \subset X$ with the following properties.

\textup{(1)} The set $A:=\bigcup _{\mid t\mid \leq T,\mid n\mid \leq N}R_{(t,n)}(Y)$ is Borel measurable and satisfies $\mu (A)>1-\epsilon $.

\textup{(2)} There is a Borel isomorphism $F:A \cong Y\times [-T,T]\times \{ -N,\cdots ,N\}$ and a Borel measure $\nu $ on $Y$ such that 
\[ \mu F^{-1}=\nu \otimes \mathrm{Lebesgue \ measure} \otimes \mathrm{counting \ measure}. \]

\textup{(3)} Under this identification, we have 
  \[ R_{(t,n)}(y,s,m)=(y,s+t,m+n)\] 
for $y\in Y$, $|s+t|\leq T$, $|s|\leq T$, $m\in \{ -N,\cdots,N\}$, $|m|\leq N$, $|n+m|\leq N$.
\label{3.1.2}
\end{lemm} 

Now, we give the characterization of the Rohlin property for flows constructed in Lemma \ref{3.1.1}.

\begin{theo}
\label{17}

For flows constructed in Lemma \ref{3.1.1}, consider the following five conditions.

\textup{(1)} The flow $\theta $ has the Rohlin property.

\textup{(2)} The action $\{(\theta_t|_D) \circ \alpha _n \}_{(t,n)\in \textbf{R}\times \textbf{Z}}$ is faithful on $D$.

\textup{(3)} The flow $\theta $ is centrally free. That is, $\theta _t$ is free on $M_\omega $ for $t\not =0$.

\textup{(4)} The flow $\theta $ is centrally free and has full Connes spectrum.

\textup{(5)} We have $(M\rtimes _{\theta} \textbf{R})\cap M'=\textbf{C}$.

Then we have implications \textup{(1)} $\Leftrightarrow$ \textup{(2)} $\Leftrightarrow $ \textup{(3)} $\Leftrightarrow $ \textup{(4)} $\Rightarrow $ \textup{(5)}.
\label{3.1.3}
\end{theo}

\begin{proof}
The implications (1) $\Rightarrow$ (4) $\Rightarrow$ (3) and (1) $\Rightarrow$ (5) follow from Masuda--Tomatsu \cite{MT}. Hence it suffices to show the implications (3) $\Rightarrow $(2) and (2) $\Rightarrow$ (1). 

First, we show the implication (3) $\Rightarrow$ (2). Assume that condition (2) does not hold. Then  there exists $(t,n)\not =0$ such that $\theta_t=\alpha _n$. We have $t\not =0 $ because $\alpha $ is ergodic. Hence for $x\in M_\omega \subset D^{\omega }$ (as explained in subsection \ref{2.4}, the implication $M_\omega \subset D^\omega $ is shown in Theorem 8 of Ueda \cite{U2}), we have $\theta _t(x)=\lambda _{a^n}x\lambda _{a^{-n}}=x$, which implies that condition (3) does not hold.

Next, we show the implication (2) $\Rightarrow$ (1). Suppose that the action $\{(\theta_t|_D) \circ  \alpha _n \}_{(t,n)\in \textbf{R}\times \textbf{Z}}$ is  faithful. Fix $n\in \textbf{N}$. It is enough to construct a sequence $\{ u_n \}$ of unitary elements of $D$ such that 

(i) $\mu \left (|\theta _t(u_n)-e^{-ipt}u_n|^2\right )<n^{-2}$ for $|t|<n$,

(ii) $\mu \left (|\alpha(u_n)-u_n|^2\right )<n^{-2}$.

Assume that we have these $u_n$'s. Then by condition (ii), $\{ u_n\}$ asymptotically commutes with $\lambda_a$ and $\lambda _b$. Hence $\{ u_n \}$ is a centralizing sequence. By using condition (i), we have $\theta_t\left (\{u_n \}\right )=e^{-ipt}\{u_n \}$ for $t\in\textbf{R}$.                                                            

Now, we show the existence of the above $\{u_n\}$. Regard $D$ as $L^\infty (X, \mu )$, where $(X, \mu)$ is a standard probability measured space and let $S:\textbf{R}\curvearrowright (X,\mu )$ and $T:\textbf{Z}\curvearrowright (X,\mu )$ be actions induced by $\theta $, $\alpha $, respectively.
By using Lemma \ref{3.1.2} for $T:=8n^3$, $N:=8n^2$, $\epsilon :=1/8n^2$,  $R_{(s,m)}:=S_sT_m$, there exists a Borel subset $Y\subset X$ satisfying the conditions in Lemma \ref{3.1.2}.

Set
\begin{equation*}
     u_n(y,s,m):= \begin{cases}
                    e^{ips} & \text{for $(y,s,m)\in A$}, \\
                    1       & \text{for $ x\in X \backslash A$}.
                   \end{cases} 
\end{equation*}
Then by condition (3) of Lemma \ref{3.1.2}, we have 
\[ (\theta_t(u_n)-e^{-ipt}u_n)(x)=0 \text{ for $ x\in \{ (y,s,m)\in A |\ |s|\leq T-n \}$}. \]
Hence we have 
\begin{equation*}
  \begin{split}
    \mu (|\theta _t(u_n)-e^{-ipt}u_n|^2) &\leq 4\mu (X\backslash \{ (y,s,m)\in A|\ |s|\leq T-n\}) \\
                                         &=4(\mu (X\backslash A)+\mu(\{(y,s,m)\in A|\ |s|>T-n\})) \\
                                         &\leq 4(\epsilon +n/T)=n^{-2}. 
  \end{split}
\end{equation*}
By similar computation to this, we have $\mu (|\alpha(u_n)-u_n|^2)<n^{-2}$.
\end{proof}
By this theorem, it is possible to see that there exist Rohlin flows on the factor $M$. In order to do this, it is important to note that if an action $\beta :\mathbf{Z}\curvearrowright D$ is free ergodic probability measure preserving, then $D\rtimes _{\beta *\beta } \mathbf{F}_2$ is isomorphic to $M$ considered in this section, which is shown by Connes--Feldman--Weiss \cite{CFW} and the uniqueness of the amalgamated free product. 
\begin{exam}
Let $(\tilde{D}, \mu)$ be a diffuse separable abelian von Neumann algebra with a normal faithful trace and let $\tilde{\theta }$ be a $\mu$-preserving faithful flow on $D$. Set $D:=\otimes _{n=-\infty}^{\infty}(\tilde{D},\mu )^n$ and $\alpha:\mathbf{Z}\curvearrowright D$ be a Bernoulli shift. Then the diagonal action $\theta : \mathbf{R} \curvearrowright D$ of $\tilde{\theta}$ extends to $M:=D\rtimes _\alpha \mathbf{F}_2$ and has the Rohlin property.
\end{exam}
Other examples are given in the following.
\begin{exam}
Let $D=L^{\infty }(\textbf{T}^2)$($=L^\infty ((\mathbf{R}/\mathbf{Z})^2)$) and let $\alpha :\textbf{Z}\curvearrowright D$ be an action defined by $\alpha (f)(r,s)=f(r-1/\sqrt{2}, s-1/\sqrt{3})$
 for $(r,s)\in \textbf{T}^2$, $f\in D$. 
Then $D\rtimes _{\alpha*\alpha}\textbf{F}_2$ is isomorphic to $M$. By Lemma \ref{3.1.1}, we can define a flow $\theta^{\lambda, \mu, p, q}:\textbf{R}\curvearrowright D\rtimes _{\alpha *\alpha }\textbf{F} _2$ by 
\[ \theta^{\lambda , \mu , p, q}_t(f)(r,s)=f(r-pt,s-qt) \]
 for $(r,s)\in \textbf{T}^2$, $f\in D$, $t\in \textbf{R}$,
\[ \theta^{\lambda , \mu , p, q}_t(\lambda _a)=e^{i\lambda t}\lambda _a,\  \theta^{\lambda , \mu , p, q}_t(\lambda _b)=e^{i\mu t}\lambda _b\]
 for $t\in \textbf{R}$.
This $\theta^{\lambda , \mu , p, q}$ has the Rohlin property if and only if $(p,q)\not =r(n/\sqrt{2}-m, n/\sqrt{3}-l)$ for any  $r\in \textbf{R}$, $n,m,l\in \textbf{Z}$. 
\label{3.1.4}
\end{exam}
\begin{proof}
In order to show this, by Theorem \ref{17}, it is enough to show that the action $\{ (\theta ^{\lambda, \mu, p,q}_t|_D)\circ \alpha_n\} $ is faithful if and only if the above condition holds.  For $(t,n)\in \textbf{R} \times \textbf{Z}$, $\theta ^{\lambda, \mu, p,q}_t|_D=\alpha_n$ if and only if $pt=n/\sqrt{2}+m$, $qt= n/\sqrt{3}+l$ for some $m,l\in \textbf{Z}$. Hence $\{( \theta ^{\lambda, \mu, p,q}_t|_D)\circ \alpha_n\} $ is faithful if and only if $(p,q)\not=(n/\sqrt{2}+m, n/\sqrt{3}+l)/t$ for all $t\in \textbf{R}\setminus \{0\}, n,m,l\in \textbf{Z}$. 
\end{proof}
If we further assume that $(p,q)\not =r(s/\sqrt{2}-m, s/\sqrt{3}-l)$ for any  $r,s \in \textbf{R}$, $m,l\in \textbf{Z}$, then this also gives a new example of a Rohlin flow on the $C^*$-algebra $C(\mathbf{T}^2)\rtimes _{\alpha *\alpha}\mathbf{F}_2$, which is shown by the same argument as in Proposition 2.5 of Kishimoto \cite{Ksm}.

\begin{rema}
Let $\alpha :G\curvearrowright D$ be a non-singular free ergodic action of a discrete group. If the action $\alpha $ is stable (See Definition 3.1 of Jones--Schmidt \cite{JS}), then the factor $M:=D\rtimes _{\alpha *\alpha } (G*G)$ admits Rohlin flows. This is shown by the argument similar to (2) $\Rightarrow$ (1) of Theorem \ref{3.1.3}. In particular, by Corollary 5.8 of Ueda \cite{U1}, for any $\lambda \in [0,1]$, there exists a type $\mathrm{III}_\lambda $ non-McDuff factor which admits Rohlin flows. 
\end{rema}

\section{Classification of Rohlin flows}
\label{sec4}

In section 3, we have shown that there are Rohlin flows on non-McDuff factors. So it is natural to try to classify them. In this section, first, we apply Masuda--Tomatsu's theorem for our Rohlin flows, and we obtain a classification up to strong cocycle conjugacy. However, strong cocycle conjugacy is nothing but a technical classification, and a main purpose is a classification up to usual cocycle conjugacy. So it is important to know whether these two classifications are different or not, and we show that they are strictly different for our flows.

\bigskip

The following lemma is useful for investigating whether the difference of two flows is approximately inner or not. Let $M=A*_DB$ be an amalgamated free product over a common Cartan subalgebra $D$, as in subsection \ref{2.4}. Let $\mu $ be a normal faithful state on $D$ and let $E_A:A\to D$, $E_B:B\to D$, $E:M\to D$ be the conditional expectations as in subsection \ref{2.4}. 

\begin{lemm}\textup{(Lemma 2.1 of Popa \cite{P}, Theorem 5 of Ueda \cite{U2})}

Let $M=A*_DB$, $\mu $, $E_A$, $E_B$, $E$ be as above. Let $x\in M^\omega $ and let $v$, $w$ be unitaries of $A$ with $\mu \circ E_A (u^* \cdot u)=\mu \circ E_A =\mu \circ E_A (v^* \cdot v)$. Assume that
$E_A(v^n)=0$, $E_A(w^n)=0$\ \textup{(}$n\not=0$\textup{)}, $vDv^*=D=wDw^*$, $x=vxw^*$. Then for $y_1,y_2\in \ker E_B$, we have 
\[ \parallel y_1x-xy_2 \parallel_{(\mu \circ E)^\omega }^2  \geq \parallel y_1(x-E^\omega (x)) \parallel_{(\mu \circ E)^\omega }^2 +\parallel (x-E^{\omega }(x))y_2\parallel _{(\mu \circ E)^\omega }^2 ,\]
where $(\mu \circ E)^{\omega }:M^{\omega }\to \mathbf{C}$, $E^\omega :M^\omega \to D^\omega $ are maps induced by $\mu \circ E$ and $E$, respectively \textup{(see subsection 2.2 of Ueda \cite{U2})} and $\parallel x\parallel _{(\mu \circ E)^\omega}=((\mu \circ E)^\omega (x^*x))^{1/2}$ for $x\in M^\omega $. 
\label{4.1.1}
\end{lemm}

By using this lemma, it is possible to show the following lemma, which is crucially important to investigate the approximate innerness of flows. Let $M=A*_DA$ be the type $\mathrm{II}_1$ amalgamated free product factor considered in section \ref{sec3}.

\begin{lemm}
Let $\theta$ be an automorphism of $M=A*_DA$ which globally preserves $D$ and satisfies $\theta (\lambda_a)=u^1\lambda_a$, $\theta (\lambda_b)=u^2\lambda_b$ for some $u^1$, $u^2$ $\in U(D)$. Then the automorphism $\theta$ is approximately inner if and only if $\theta |_D=\mathrm{id}$, $u^1=u^2$.
\label{aaaa}
\end{lemm}
\begin{proof}
 This is shown in the proof of Theorem 14 of Ueda \cite{U2} in a more general setting. Here we give a proof briefly.

First, we show the ``only if'' part. Assume that $\theta $ is approximately inner. Then there exists a unitary $\{u_n\}$ of  $M^{\omega }$ such that $\theta (y)=\mathrm{strong}$-$\lim_{n\to \omega}u_n^* y u_n$ for $y\in M$. Then by using Lemma \ref{4.1.1} for $v=\lambda _a$, $w={u^1}^*\lambda_a$, $y_1=\lambda _b$, $y_2={u^2}^*\lambda _b$, $x=\{ u_n\}$, we have $\{u_n\} -E^\omega (\{ u_n \})=0$. Hence we have $\{u_n\} \in D^\omega $, which implies that $\theta |_D=\mathrm{id}$, and we have 
\[ u^1=\theta (\lambda _a)\lambda _a^*=\lim _{n\rightarrow \omega }u_n \lambda _a u_n^*\lambda _a^*=\lim _{n\rightarrow \omega }u_n\alpha (u_n^*)=\lim _{n\rightarrow \omega }u_n\lambda _b^*u_n^* \lambda _b =u^2.\]

Next, we show the ``if'' part. Assume that $\theta |_D=\mathrm{id}$, $u^1=u^2$. We construct a sequence $\{ u_n\}$ of unitaries of $D $ such that $u_n\alpha (u_n^*)\rightarrow u^1$. By using the Rohlin lemma for $\alpha $, there exists a partition $\{e_k\}_{k=0}^{n}\subset \mathrm{Proj}(D)$ of unity in $D$ such that 
\[\alpha (e_k)=e_{k+1} \  \text{for} \ k=1,\cdots ,n-1, \ \mu (e_0)<1/(n+1). \]
 Set 
\[ u_n=\sum _{k=0}^n v_k e_k, \\ v_1=u^1,  \\ v_{k+1}=\alpha (v_k)u^1. \]
for $k=1, \cdots , n-1$. Note that 
\[ v_ke_k\alpha (e_lv_l^*)=v_kv_{l+1}^*u^1e_ke_{l+1}=\delta _{k,l+1}u^1e_k \]
 for $k,l =1, \cdots n-1$. Hence by a similar computation to the one in the proof of Theorem \ref{3.1.3}, we have  $u_n\alpha (u_n^*)\to u^1$. Hence we have 
\[ u_n\lambda _a u_n^* \to u^1\lambda _a= \theta (\lambda_a), \]
\[ u_n\lambda _b u_n^*\to u^1\lambda _b=u^2\lambda _b=\theta(\lambda _b), \]
 which implies that  $\mathrm{Ad}u_n(y)\rightarrow \theta (y)$ strongly for $y\in M$. 
 \end{proof}
 
 \begin{lemm}
The Rohlin flows constructed in Theorem \ref{3.1.3} are completely classified by $\{\theta |_D ,u^1_t{u^2_t}^* \}$, up to strong cocycle conjugacy.
\label{4.1.2}
\end{lemm}

\begin{proof}
This lemma immediately follows from Theorem \ref{2.3.4} and Lemma \ref{aaaa}.   
\end{proof}

\begin{exam}
The Rohlin flows considered in Example \ref{3.1.4} are completely classified by $(p, q, \lambda -\mu)$, up to strong cocycle conjugacy.
\label{4.1.3}
\end{exam}

The following theorem is the main theorem of this section.

\begin{theo}
For Rohlin flows in Example \ref{3.1.4}, usual cocycle conjugacy and strong cocycle conjugacy are different.
\label{4.2.2}
\end{theo}

The following lemma is an essential part of Theorem \ref{4.2.2}. Recall that the discrete spectrum $\mathrm{Sp}_d(\theta )$ of a flow $\theta $ on a von Neumann algebra $M$ is the set
\[ \mathrm{Sp}_d (\theta ):=\{ p\in \mathbf{R} \mid \mathrm{there} \ \mathrm{exists} \ x\in M\setminus \{ 0\} \ \mathrm{with} \ \theta _t(x) =e^{ipt}x \ \mathrm{for} \ t\in \mathbf{R}\} . \] 

\begin{lemm}
\label{bbbb}
 Let   $\theta^{\lambda_1 , \mu_1 , p_1, q_1}$, $\theta^{\lambda _2 , \mu _2 , p _2, q _2}$ be two Rohlin flows mentioned in Example \ref{3.1.4}. Then they are cocycle conjugate if there exist $r\in \textbf{R}$ and two points  $c$, $d$ of $\mathrm{Sp}_d (\theta ^{\lambda_1 , \mu_1 , p_1, q_1}\mid _D)$ such that one of the following conditions holds.

\textup{(1)} We have $(p_1,q_1)=(p_2,q_2)$ and 
\[\left(
   \begin{array}{cccc}
    \lambda _2 \\
    \mu _2 
   \end{array}
   \right) =\left(
   \begin{array}{cccc}
    \lambda _1 \\
    \mu _1 
   \end{array}
   \right) +\left(
   \begin{array}{cccc}
    r  \\
    r
   \end{array}
   \right) +\left(
   \begin{array}{cccc}
   c \\
   d
   \end{array}
   \right).\]
   
\textup{(2)} We have $(p_1,q_1)=-(p_2,q_2)$ and 
\[\left(
   \begin{array}{cccc}
    \lambda _2 \\
    \mu _2 
   \end{array}
   \right) =\left(  
          \begin{array}{@{\,}cccc@{\,}}
           0 & -1 \\
           1 & -2 
          \end{array}
         \right)\left(
   \begin{array}{cccc}
    \lambda _1 \\
    \mu _1 
   \end{array}
   \right) +\left(
   \begin{array}{cccc}
    r  \\
    r
   \end{array}
   \right) +\left(
   \begin{array}{cccc}
   c \\
   d
   \end{array}
   \right).\]
   \end{lemm}
 
 \begin{proof}
 Assume that one of the above conditions holds.  First, consider the case when condition (2) holds. Let $\sigma $ be an automorphism of $D$ defined by
\[ \sigma (f)(s,t):=f(-s,-t) \]
for $f\in D$, $(s,t) \in \mathbf{T}^2$.  We show the following claim.

\textbf{Claim.}
The automorphism $\sigma $ extends to an automorphism of $M$ by 
\begin{align*}
 \sigma(\lambda _a)=& \lambda _{b^{-1}}, &\sigma(\lambda _b)=&\lambda_{ab^{-2}}. 
\end{align*}
\textit{Proof of Claim.}
 Set an automorphism $\beta $ on $D$ by $\beta :=\alpha ^{-1}$.
 Then we have  $\sigma \circ \beta \circ \sigma ^{-1} =\alpha $. Hence by Lemma 7.5 of Takesaki \cite{T1}, there exists an isomorphism $\pi _A : D\rtimes _{\alpha *\alpha }\{a\}^{\mathbf{Z}} \cong D\rtimes _\alpha \mathbf{Z} \to D\rtimes _\beta \mathbf{Z}\cong D\rtimes _{\alpha *\alpha }{\{b^{-1}\}}^{\mathbf{Z}}$ satisfying 
\[ D\ni f \mapsto \sigma (f), \\\ \lambda _{a} \mapsto \lambda _{b^{-1}}. \]
 Similarly, there exists an isomorphism $\pi _B : D\rtimes _{\alpha *\alpha}\{ b\}^{\mathbf{Z}} \cong D\rtimes _\alpha \mathbf{Z} \to D\rtimes _\beta \mathbf{Z}\cong D\rtimes _{\alpha *\alpha }{\{ab^{-2}\}}^{\mathbf{Z}}$ satisfying 
\[ D\ni f \mapsto \sigma (f),\ \lambda _{b} \mapsto \lambda _{ab^{-2}}. \] 
Note that the endomorphism $\rho $ of $\mathbf{F}_2$ defined by $a\mapsto b^{-1}$, $b\mapsto ab^{-2}$ is bijective. The inverse is given by $a\mapsto ba^{-2}$, $b\mapsto a^{-1}$. By the injectivity of $\rho$, the images of $\pi _A$ and $\pi _B$  are free over $D$ (see condition (4) of subsection \ref{2.4}). By this observation, it is easy to see that $\pi _A$ and $\pi _B$ satisfy conditions (1)--(4) of subsection \ref{2.4}. Thus by the uniqueness of the amalgamated free product, $\sigma $ extends to an automorphism of $M$.
\qed

Now we continue the proof of the Lemma. Since $\sigma ^{-1}\circ \theta ^{\lambda_1 , \mu_1 , p_1, q_1}  \circ \sigma = \theta ^{-\mu_1 , \lambda _1-2\mu_1 , -p_1, -q_1} $, by replacing $\theta ^{\lambda_1 , \mu_1 , p_1, q_1}$ by $\sigma ^{-1}\circ\theta ^{\lambda_1 , \mu_1 , p_1, q_1}\circ \sigma  $, it is enough to consider the case when condition (1) holds. Assume that condition (1) holds. Since $c\in \mathrm{Sp}_d(\theta ^{\lambda_1 , \mu_1 , p_1, q_1} \mid _D)$, there exists $u\in D$ such that  $\parallel u\parallel =1$ and $ \theta ^{\lambda_1 , \mu_1 , p_1, q_1}_t(u)=e^{ict}u$ for $t\in \mathbf{R}$. Since $u^*u$($=uu^*$) is fixed by $\theta ^{\lambda_1 , \mu_1 , p_1, q_1}$, $u^*u=uu^*=1$ by the ergodicity of $\theta ^{\lambda_1 , \mu_1 , p_1, q_1}\mid _D$. Similarly, there exists a unitary $v$ of $D$ with $\theta ^{\lambda_1 , \mu_1 , p_1, q_1}_t(v)=e^{idt}v$ for $t\in \mathbf{R}$. 

Then the identity map $\sigma $ of $D$ extends to $M$ by $\sigma (\lambda _a)=u\lambda _a$, $\sigma (\lambda _b)=v\lambda _b$. By replacing $\theta ^{\lambda_1 , \mu_1 , p_1, q_1}$ by $\sigma ^{-1} \circ\theta ^{\lambda_1 , \mu_1 , p_1, q_1}\circ \sigma $, we may assume that $c=d=0$. Hence by using Example \ref{4.1.3}, $\theta^{\lambda_1 , \mu_1 , p_1, q_1}$ and $\theta^{\lambda _2 , \mu _2 , p _2, q _2}$ are cocycle conjugate. 
 \end{proof}
 Now, we return to the proof of Theorem \ref{4.2.2}.
 
 \bigskip
 
 \textit{Proof of Theorem \ref{4.2.2}}.
Let $\theta^{\lambda_1 , \mu_1 , p_1, q_1}$ and $\theta^{\lambda _2 , \mu _2 , p _2, q _2}$ be two Rohlin flows considered in Example \ref{3.1.4}. Then by Example \ref{4.1.3}, they are strongly cocycle conjugate if and only if  $\lambda _1-\mu _1=\lambda _2-\mu _2$,  $p_1=p_2$ and $q_1=q_2$. On the other hand, by Lemma \ref{bbbb}, they are cocycle conjugate if $(p_2,q_2)=(-p_1,-q_1)$ and $(\lambda _2, \mu _2)=(-\mu _1, \lambda _1-2\mu _1)$.
\qed

\bigskip
  
We interpret this difference comes from the fact that $\overline{\mathrm{Int}}(M)$ is too small. For our flows, strong cocycle conjugacy is just perturbations by $\overline{\{ \mathrm{Ad}(u)\mid u\in U(D)\}}$. Consequently, even though Masuda--Tomatsu's analysis is applicable for flows on non-McDuff factors, they are not so powerful and another method is needed to analyze flows on non-McDuff factors.

\textbf{Acknowledgements.}
The author is thankful to Professor Reiji Tomatsu, who raises his interest in this problem and gives him  very useful advice. The author would also thank Professor Yasuyuki Kawahigashi, who is his advisor and Professor Toshihiko Masuda for carefully reading his paper and for giving him valuable comments. He would like to thank the referee for advice which greatly improves the presentation of the paper. He is supported by Leading Graduate Course for Frontiers of Mathematical Sciences and Physics of the University of Tokyo.

\end{document}